\documentclass[12pt]{article}

\usepackage{amsmath}
\usepackage{amsfonts}
\usepackage{amsthm}
\usepackage{mathrsfs}
\usepackage[export]{adjustbox}
\usepackage{graphicx}
\graphicspath{ {./images/} }

\newtheorem{theorem}{Theorem}[section]

\newtheorem{definition}[theorem]{Definition}

\def\R{\mathbb R}

\def\<{\langle}
\def\>{\rangle}

\def\dist{{\rm dist}}

\newcommand{\romd}{\mathrm{d}}

\begin{document}
\title{Some comments on Laakso graphs and sets of differences}

\author{Alexandros Margaris and James C. Robinson\\
   Mathematics Institute\\ University of Warwick, Coventry\\ CV4 7AL, UK}

\maketitle

\begin{abstract}
We recall a variation of a construction due to Laakso \cite{LA}, also used by Lang and Plaut \cite{LA} of a doubling metric space $X$ that cannot be embedded into any Hilbert space. We give a more concrete version of this construction and motivated by the results of Olson \& Robinson \cite{OR}, we consider the Kuratowski embedding $\Phi(X)$ of $X$ into $L^{\infty}(X)$ and prove that $\Phi(X)-\Phi(X)$ is not doubling. 
\end{abstract}

\section{Introduction}

We say that a metric space $(X,d)$ is doubling, with doubling constant $K$, if for every given $x \in X$ and $r > 0$, there exist $y_{1},...,y_{K}$ in $X$ such that 
$$B(x,r) \subset \bigcup_{i=1}^{K} B(y_{i}, r/2).$$ We say that a metric space $(X,d)$ embeds into a normed space $\left(Y, \|\cdot\|\right)$ in a  bi-Lipschitz way if there exists $f \colon X \to Y$ and some constant $L >0$, such that for all $x,y \in X$
$$d(x,y) \frac{1}{L} \leq \|f(x) - f(y)\| \leq L \, d(x,y).$$
We know that when a metric space embeds into an Euclidean space in a bi-Lipschitz way then it must be doubling, but there are examples due to Laakso \cite{LA}, Lang \& Plaut \cite{LP} and Semmes \cite{SE} that show that this condition is not sufficient. 

In 1983, Assouad \cite{Ass} proved that any doubling metric space $X$ can be embedded into an Euclidean space in a bi--H\"older way, i.e. for any $0 <\alpha <1$, there exists some $k \in \mathbb{N}$ and a $\phi \colon X \to \mathbb{R}^{k}$ such that
$$\frac{1}{L} d(x,y)^{\alpha} \leq |\phi(x) - \phi(y)| \leq L d(x,y)^{\alpha},$$
for all $x,y \in X$.

We also want to recall the notion of a homogeneous metric space. A subset $A$ of a metric space $(X,d)$ is said to be $(M,s)-homogeneous$ if for every $x \in A$ and $r > \rho > 0$
$$N(A \cap B(x,r), \rho) \leq M\left(\frac{R}{r}\right)^{s},$$
where $N(A \cap B(x,r), \rho)$ denotes the minimum number of balls of radius $\rho$ required to cover $A \cap B(x,r)$.
It can be easily shown that a metric space is homogeneous if and only if it is doubling (see Chapter 9 in the book of Robinson \cite{JCR}). In 2010, Olson and Robinson \cite{OR} used Assouad's construction \cite{Ass} and proved that if $X$ is homogeneous then it can be also embedded in an almost bi--Lipschitz way into a Hilbert space $H$, i.e. for any $\gamma >\frac{1}{2}$, there exists a map $f \colon (X,d) \to H$ such that for some positive constant $L$
$$\frac{1}{L} \frac{d(x,y)}{|\log d(x,y)|^{\gamma}} \leq \|f(x) - f(y)\| \leq L \, d(x,y).$$ Moreover, in 2014, Robinson \cite{JCR1} generalised the above result and and proved the following embedding theorem for subsets of Banach spaces.
\begin{theorem}\label{ET1}
Suppose $X$ is a compact subset of a real Banach space $\mathfrak{B}$ such that the set $X-X = \{x-y : x,y \, \in \, X\}$ is homogeneous. Then for any $\gamma >1$, there exists a natural number $N$ and a dense set of linear maps $L \colon \mathfrak{B} \to \mathbb{R}^{N}$, that are injective on $X$ and $\gamma$-almost bi--Lipschitz, i.e. for some constant $c_{L}$
$$\frac{1}{c_{L}} \frac{\|x-y\|}{|\log \|x-y\||^{\gamma}} \leq |L(x) - L(y)| \leq c_{L} \, \|x-y\|,$$
for all $x,y \in X$.
\end{theorem}
The above theorem can be used to provide embeddings of subsets of compact metric spaces, using the isometric embedding $\Phi \colon (X,d) \to L^{\infty}(X)$, given by $x \mapsto d(x, \cdot),$ due to Kuratowski (see the notes from Heinonen \cite{HJ} for a detailed proof). In particular, we can define $`X-X$' in this context to mean
$$X-X \equiv \Phi(X)-\Phi(X) =  \{f \in L^{\infty}(X) : f = d(x,\cdot) - d(y,\cdot), \, \, \text{for} \, \, x, y \, \, \text{in} \, \, X\}.$$

It is an open problem whether we can prove almost bi--Lipschitz embeddings into Euclidean spaces under a weaker condition than the one in Theorem \ref{ET1}. It is also a question whether we can prove `better' embeddings than almost bi--Lipschitz when $X-X$ is homogeneous. There are known examples of homogeneous subsets of Banach spaces for which the set of differences is not homogeneous (see for example Chapter 9 in the book of Robinson \cite{JCR}) but there is no information on the embedding properties of these sets.

In this paper, we consider a variation of the construction due to Laakso \cite{LA}, which was used by Lang \& Plaut \cite{LP} to construct a doubling metric space $X$ that cannot be embedded in a bi--Lipschitz way into any Hilbert space. We prove that $\Phi(X) - \Phi(X)$ is not doubling as a subset of $L^{\infty}(X)$, thus giving motivation towards the direction of studying the set of differences more closely. In this way, we also rule out the possibility of using Theorem \ref{ET1} to prove the existence of almost bi--Lipschitz maps into a Euclidean space for this set.

\section{The Laakso graphs}

We first recall the definition of the Gromov--Hausdorff distance of compact metric spaces.

\begin{definition}
Suppose $X,Y$ are compact metric spaces. Then
$$d_{GH}(X,Y) = \inf \max\{\dist(f(X),g(Y), \dist(g(Y),f(X))\},$$
where the infimum is taken over all metric spaces $M$ and all possible isometric embeddings $f \colon X \to M$ and $g \colon Y \to M$.
\end{definition}

It is easy to check that $d_{GH}(X,Y)=0$ if and only if $X$ is isometric to $Y$ (see the lecture notes from Heinonen \cite{HJ} for a proof), proving that the set of all isometry classes of compact metric spaces equipped with $d_{GH}$ forms a metric space, which is compact (see Heinonen \cite{HJ} again).

We now recall the construction due to Lang and Plaut \cite{LP} of a metric space that is homogeneous but does not embed in a bi--Lipschitz way into any Hilbert space. Here we use a construction that is somewhat more concrete than that of Lang and Plaut \cite{LP}. We define the limiting metric space explicitly and then prove that it coincides with the one defined by Lang and Plaut.

Let $X_{0}$ be a single edge of length $1$. To construct $X_{i+1}$ from $X_{i}$, we take six copies of $X_{i}$ and rescale them by the factor of 
$\frac{1}{4}$ as in the following Figure.

\begin{figure}[h]\label{F1}
\begin{center}
\includegraphics[scale=0.8]{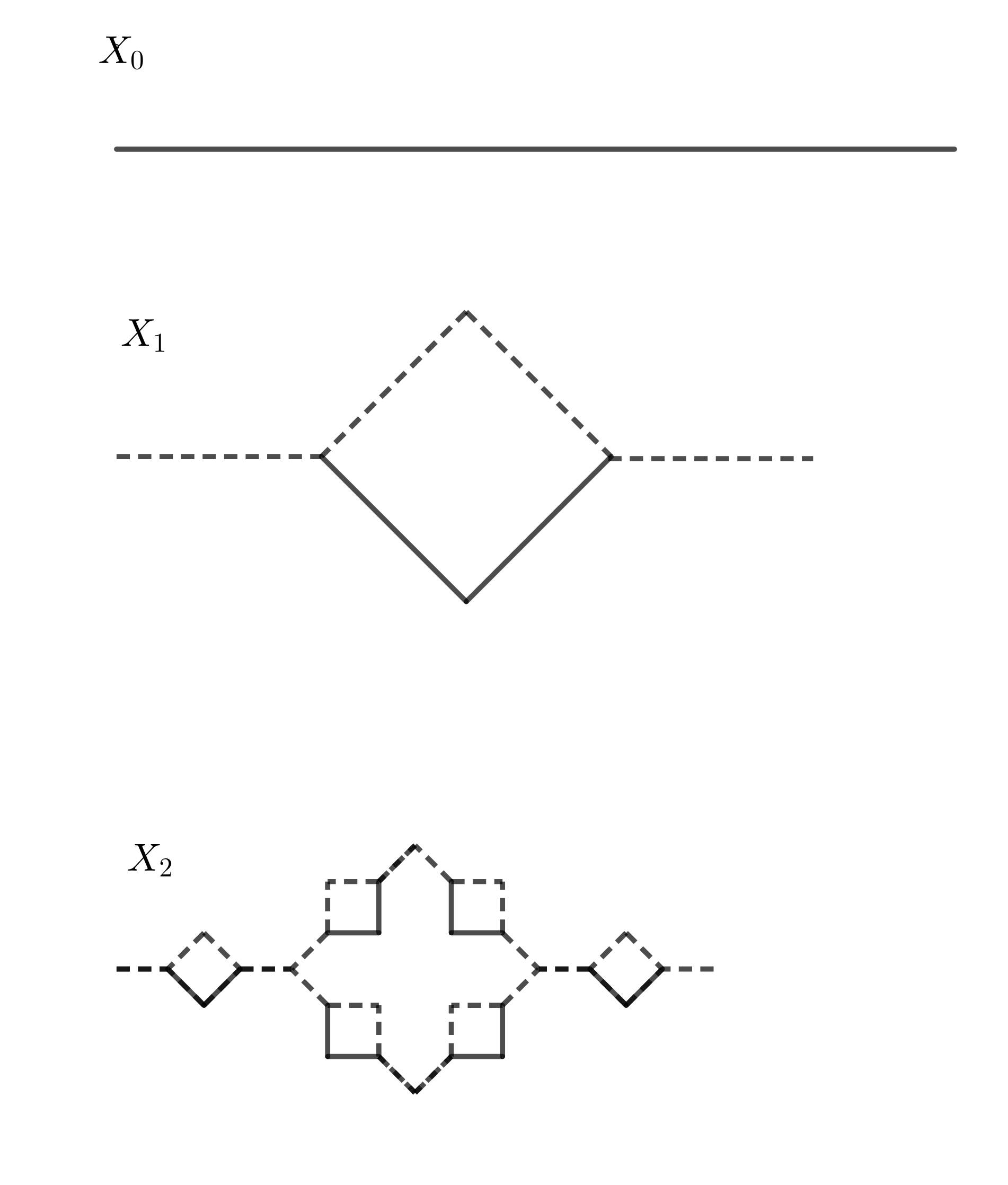}
\caption{The first stages of the construction. At each step $i$ the dotted subset is isomorphic to $X_{i-1}$.}
\end{center}
\end{figure}

We note that each $X_{i}$ has diameter $1$, has two endpoints, and comprises of $6^{i}$ edges of length $4^{-i}$ each. Every $X_{i}$, for $i > 0$ also includes $6^{i-1}$ `squares', which we call `edge cycles' for the rest of the paper. We define a metric $\varrho_{i}(x,y)$ on each of the $X_{i}$ to be the geodesic distance, i.e. the shortest path that we need to travel on the graph to get from $x$ to $y$. For any $j > i$, we construct an isometric embedding of $X_{i}$ into $X_{j}$, by identifying vertices in $X_{i}$ with vertices in $X_{j}$ and endpoints with endpoints. The image of $X_{i}$ into $X_{j}$ is represented with the dotted lines in the above figure. It is also easy to see that $d_{GH}(X_{i},X_{j}) < \left(\frac{1}{4}\right)^{i}$, and so $\{\left(X_{i}, \varrho_{i}\right)\}_{i=1}^{\infty}$ forms a Cauchy sequence in the Gromov--Hausdorff metric and it follows that it converges to a limiting metric space $X$, which is used by Lang and Plaut in their argument. We now construct this space $X$ explicitly using the following procedure

Let $(X_{i}, \varrho_{i})$ be as above for any $j >i$ let $h_{i \rightarrow j}$ denote an isometric embedding of $X_{i}$ into $X_{j}$.
Then, we take $X^{*} = \cup_{i=1}^{\infty} X_{i}$ and define a pseudometric on $X^{*}$, by setting

$$\varrho^{*}(x,y) = \begin{cases} 
\varrho_{i}(x,y) & \text{if } x,y \in X_{i}.\\
\varrho_{j}(h_{i \rightarrow j}(x),y) & \text{if } x \in X_{i} , y \in X_{j} \, \, \,  \text{and } i <j. 
\end{cases}$$

We now define a new metric space $X$, by identifying points in $X_{i}$ with their respective images in all $X_{j}$ for $j > i$. For all $x, y \in X^{*}$, we define the following equivalence relation

$$x\sim y \Leftrightarrow \varrho^{*}(x,y) = 0,$$ and we set $X = \cup_{i=1}^{\infty}[X_{i}].$
Then, for any $[x], [y] \in X$, we define
$$\varrho([x],[y]) = \varrho^{*}(x,y).$$

This definition of $X$ does not depend on the embedding we choose at each step, since if we consider another we end up with an isometric metric space.

Using the above construction, it is easy to check that
$$\romd_{GH}(X,X_{i})  \rightarrow 0.$$ 
Indeed, let $ \pi \colon X_{i} \to X$ be such that for any $x \in X_{i}$, 
$$\pi(x) = [x].$$ It is immediate that $\pi$ is an isometry from $X_{i}$ onto $[X_{i}]$.
Therefore,
$$\romd_{GH}(X,X_{i}) = \romd_{GH}(X,[X_{i}]) \leq \romd_{X}(X, [X_{i}]).$$
Let $x \in X\setminus [X_{i}]$. Then, there exists $k>i$ such that $x \in [X_{k}]$.
Then, 
$$\romd_{X}(x,[X_{i}]) = \romd_{X_{k}}(x,h_{i\rightarrow k}(X_{i})) \leq \left(\frac{1}{4}\right)^{i+1} \xrightarrow{i \rightarrow \infty} 0,$$
which proves that $X$ coincides with the metric space defined by Lang and Plaut. For the rest of the argument when we mention a point $x \in X_{i}$ we refer to the class $[x]$ with respect to the above equivalence relation.

Lang \& Plaut \cite{LP} showed that $X$ is doubling with doubling constant $6$. Using the above construction, their proof becomes somewhat more transparent. For a proof see \cite{MAR}.

Now, we recall the Kuratowski embedding  
\begin{align}\label{KUR}
\Phi \colon & X \to L^{\infty}(X)\\\nonumber
& x \mapsto \varrho (\cdot, x).
\end{align}
and we define

$$X-X = \{\varrho(\cdot,x) - \varrho(\cdot,y) : x,y \in X\}.$$

We now prove that $X-X$ is not doubling.

\begin{theorem}\label{DG}
If $X$ is the metric space defined above and $\Phi \colon X \to L^{\infty}(X)$ is the Kuratowski embedding defined in \eqref{KUR} then, $\Phi(X)-\Phi(X)$ is not doubling.
\end{theorem}

\begin{proof}
We assume that $\Phi(X) - \Phi(X)$ is doubling. 

Let $r = \left(\frac{1}{4}\right)^{i}$, for some $i \in \mathbb{N}$ and take the ball $B_{X-X} \left(0, 2r\right)$. Suppose that there exist $M$ and $g_{j} \in X-X$ that satisfy
$$B_{X-X} \left(0, 2r\right) \subset \cup_{j=1}^{M} B_{X-X} (g_{j}, r).$$ 
Now, let $f \in B_{X-X} \left(0, 2r\right)$. Then, there exist $x,y \in X$ such that

$$f(z) = \varrho(x,z) - \varrho(y,z), \, \, \forall z \in X.$$
We can easily check that $\|f\|_{\infty} = \varrho(x,y) < 2r.$
Similarly, let $t_{j}, s_{j} \in X$ such that
$$g_{j}(z) = \varrho(t_{j},z) - \varrho(s_{j},z), \, \, \forall z \in X.$$

Any time we choose $x,y \in X$ such that $\varrho(x,y) < 2r$, we obtain an element of $B_{X-X} (0,2r).$
Let $[X_{k}] \subset X$ such that 
$$t_{j}, s_{j} \in [X_{k}],$$ for all $j \leq M$.
We now have two cases

If $k \leq i$, we show that for any edge cycle in $X_{i}$, there exist copies of some $t_{j}$ or $s_{j}$ that belong to this edge cycle. 
Suppose that there exist a cycle in $X_{i}$ that does not contain any images of $t_{j},s_{j}$. Then, we choose $x,y \in X_{i}$ as in the Figure \ref{F3}, where we zoom in at that specific cycle. Then, $x,y \in X$ satisfy 
\begin{align*}
& \varrho(t_{j},x) > r, \, \forall j \leq M \\
& \varrho(s_{j},x) > r, \, \forall j \leq M \\
& \varrho(s_{j},y) > 0, \, \forall j \leq M\\
& r < \varrho(x,y) < 2r.
\end{align*}

\begin{figure}[h]\label{F3}
\begin{center}
\includegraphics[scale= 0.7]{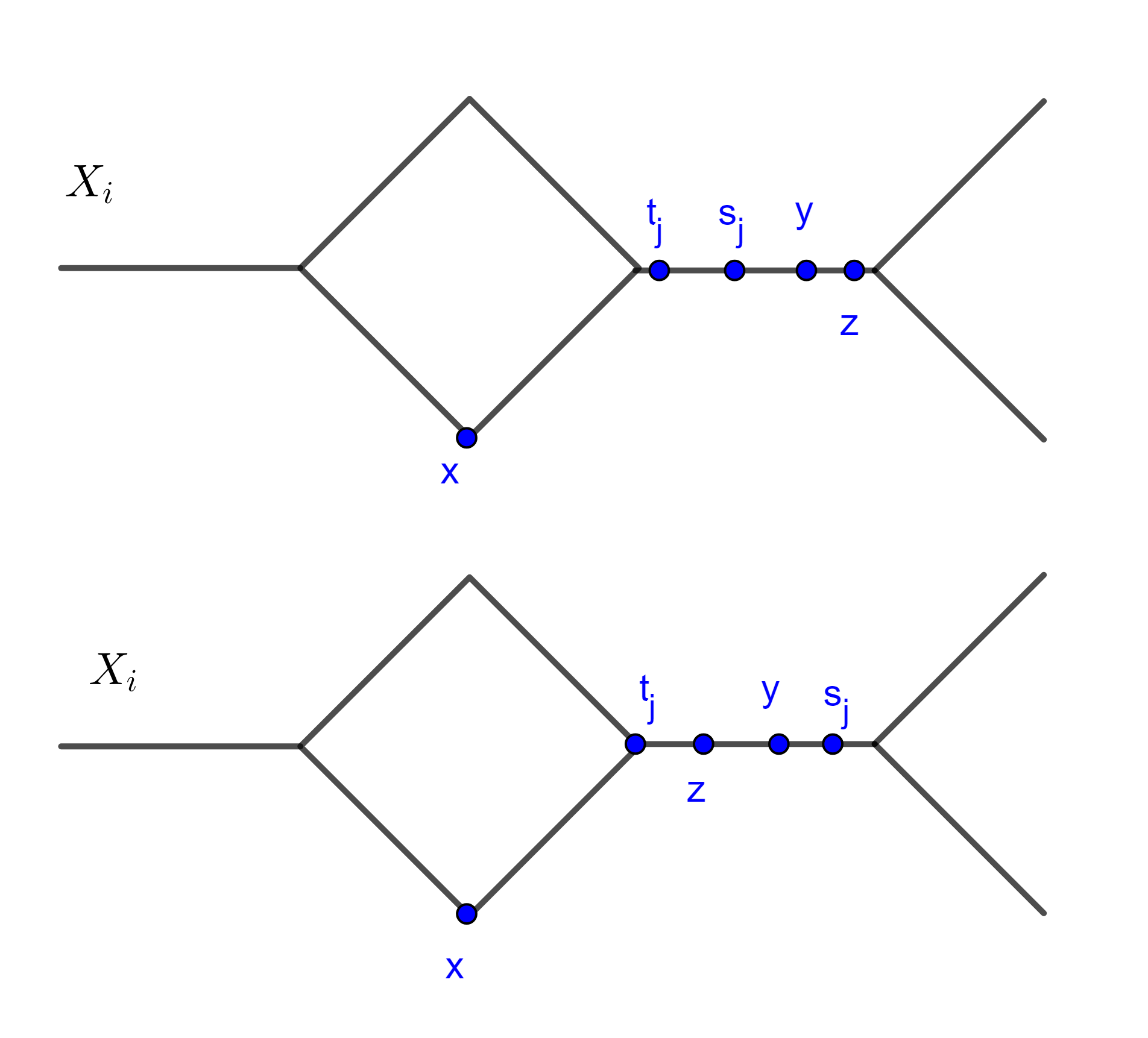}
\end{center}
\caption{The edge cycle in $X_{i}$, which does not contain any $t_{j}, s_{j}.$}
\end{figure}

Since $f \in B_{X-X}(0,2r)$, there exist $j \leq M$ such that
$$\|f-g_{j}\|_{\infty} < r \Leftrightarrow \|\varrho(x,z) - \varrho(y,z) - \varrho(t_{j},z) + \varrho(s_{j},z)\|_{\infty} < r, $$
for some $j \leq M$.
Choosing $z$ as in the above figure, depending on the position of $t_{j}, s_{j}$ we have that 
$$\|\varrho(x,z) - \varrho(y,z) - \varrho(t_{j},z) + \varrho(s_{j},z)\|_{\infty} = \varrho(t_{j},x) + \varrho(s_{j},y) >r$$ or
$$\|\varrho(x,z) - \varrho(y,z) - \varrho(t_{j},z) + \varrho(s_{j},z)\|_{\infty} = \varrho(x,y) + \varrho(s_{j},t_{j}) > r,$$ 
a contradiction.
We conclude that any edge cycle in $X_{i}$ contains one of the $t_{i}, s_{i}$ and since there are $6^{i-1}$ edge cycles contained in $X_{i}$, we deduce that 
$$M \geq 6^{i-1},$$ a contradiction.

If $k > i$, we consider the endpoints $v_{ij},u_{ij}$ that enclose an edge cycle in $X_{i}$. We rescale the cycle by the appropriate factor to create an edge cycle in $X_{k}$, with the same endpoints in $X_{k}$ (with respect to the equivalence relation we have). Since distances are preserved, we only need to repeat the above argument for all these cycles in $X_{k}$ (see also the following figure).

\begin{figure}[h]
\begin{center}
\includegraphics[scale=0.8]{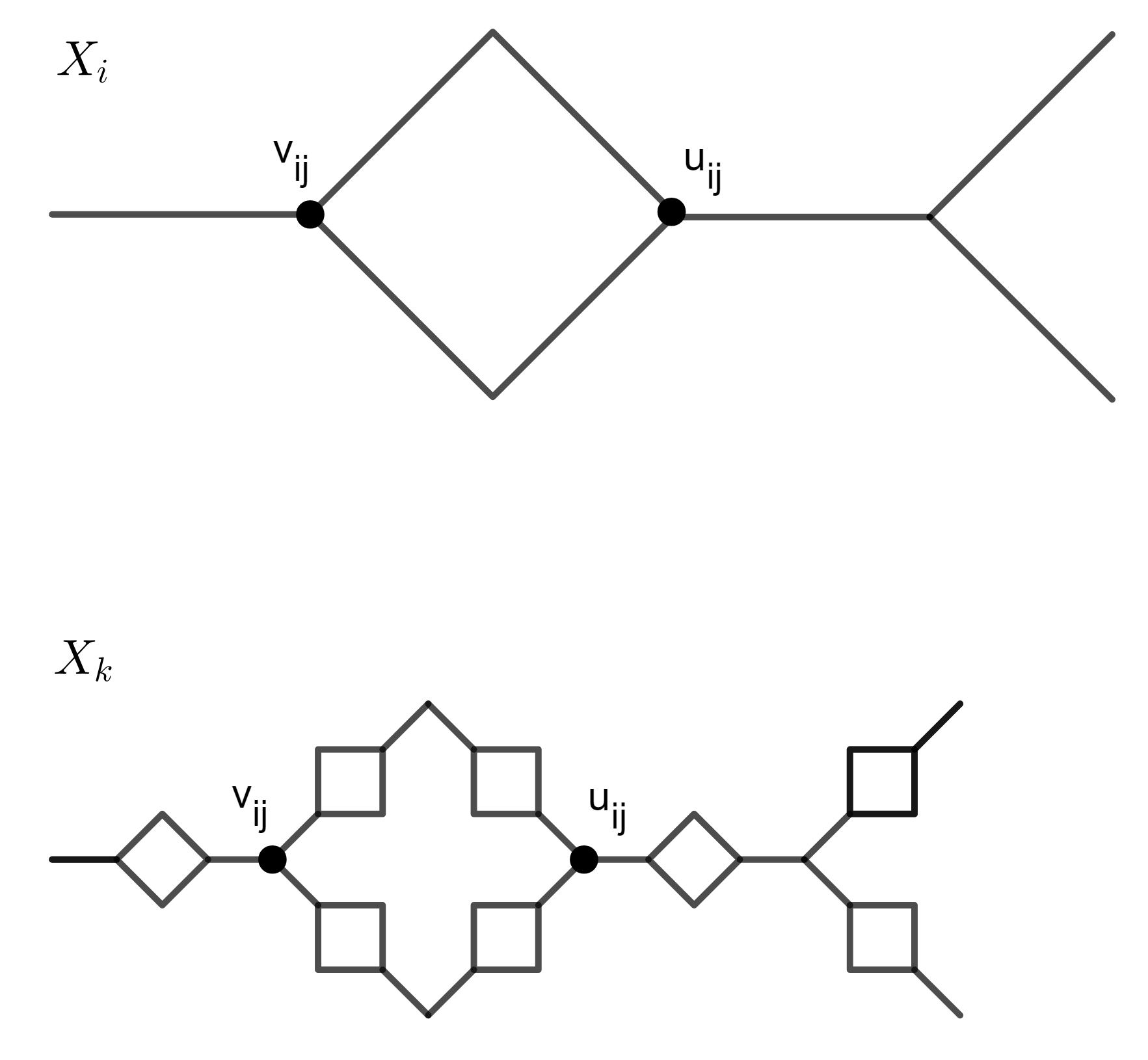}
\end{center}
\caption{The case $k > i$.}
\end{figure}

\end{proof}

\section{Conclusion}
The above result gives us an indication that we might expect better embedding properties, if we impose some condition on the set of differences. Moreover, following the results due to Olson and Robinson \cite{OR}, which are mentioned in the introduction, we arrive to the following open problems
\begin{enumerate}
\item If $X$ is as in Theorem \ref{DG}, is there an almost bi--Lipschitz embedding $f \colon X \to \mathbb{R}^{k}$?
\item Does every doubling metric space admit an almost bi--Lipschitz embedding into an Euclidean space? (A positive answer would yield a positive answer to the first question.)
\item If the set of differences $X-X$ is homogeneous, can we obtain `better' embedding properties than almost bi--Lipschitz?
\end{enumerate}


\begin{thebibliography}{99}
\bibitem{Ass} Assouad, P. `Plongements Lipschitziens dans
 $\R^n$', {\em Bull. Soc. Math. France} 111, 429-448 (1983).
\bibitem{HJ} Heinonen, J. `Geometric Embeddings of Metric Spaces.' Lectures in the Finnish Graduate School of Mathematics, University of Jyvaskyla (2003).
\bibitem{LA} Laakso, T.J., 2002. Plane with $ A_{\infty}$-weighted metric not Bi-Lipschitz embeddable to ${\mathbb{R}}^ n$. Bulletin of the London Mathematical Society, 34(6), pp.667-676.
\bibitem{LP} Lang, U. and Plaut, C., 2001. Bilipschitz embeddings of metric spaces into space forms. Geometriae Dedicata, 87(1-3), pp.285-307.
\bibitem{MAR} Margaris, A. Phd Thesis, University of Warwick, Department of Mathematics, 2019.
\bibitem{OR} Olson, E. and Robinson, J. C., 2010. Almost bi-Lipschitz embeddings and almost homogeneous sets. Transactions of the American Mathematical Society, 362(1), pp.145-168.
\bibitem{JCR1} Robinson, J. C. Log-Lipschitz embeddings of homogeneous sets with sharp logarithmic exponents and slicing products of balls. Proceedings of the American Mathematical Society. 2014;142(4):1275-88.
\bibitem{JCR} Robinson, J. C. `Dimensions, Embeddings and Attractors", Cambridge Tracts in Mathematics, Vol. 186 (2011).
\bibitem{SE} Semmes, S. (1996). On the nonexistence of bilipschitz parameterizations and geometric problems about $ A_\infty $-weights. Revista Matemática Iberoamericana, 12(2), 337-410.
\end{thebibliography}
\end{document}